\documentclass{amsart}
\usepackage[utf8]{inputenc}
\usepackage{a4wide}
\usepackage{amsmath,amsfonts,amssymb}
\usepackage{amsthm}
\usepackage{ifthen}
\usepackage{longtable}
\usepackage{array}
\usepackage{url}
\usepackage[shortlabels]{enumitem}
\usepackage{hyperref}

\hypersetup{colorlinks,
	linkcolor=green!50!black,
	citecolor=blue!70!black,
	urlcolor=red!50!black
}

\newcommand{\Z}{\mathbb{Z}}
\newcommand{\Q}{\mathbb{Q}}
\newcommand{\R}{\mathbb{R}}

\newcommand{\Ck}{C}

\newcommand{\eps}{\varepsilon}
\newcommand{\bb}{,\ldots ,}

\newtheorem*{theorem*}{Theorem}
\newtheorem{thm}{Theorem}
\newtheorem{lem}{Lemma}

\newtheorem*{claim*}{Claim}

	\newcounter{countknownthm}
	
\newtheorem{knownthm}[countknownthm]{Theorem}

\theoremstyle{definition}

	\newtheorem{rem}{Remark}
	
\usepackage{titlesec}
\titleformat{\section}
  {\center\normalsize\bfseries}{\thesection.}{1em}{}
\titleformat{\subsection}
  {\normalsize\bfseries}{\thesubsection.}{1ex}{}

\usepackage{tikz}
\usetikzlibrary{er,positioning, arrows, shapes.multipart, calc}
\usepackage{calc}

\newcommand{\mybox}[3]{
\node[entity] (step#1)[#2] {
 \begin{tabular}[t]{p{2.2cm}}
   \textbf{Step #1:} \\
   #3
\end{tabular}
 };
}

\newcommand{\pfeilrunter}[3]{
\draw [-latex,thick] (#1) --  node [fill=white, anchor=center, pos=0.5] {\small{#3}} (#2) ;
}

\newcommand{\pfeilrueber}[3]{
\draw [-latex,thick] (#1) --  node [fill=white, anchor=center, pos=0.5] {\small \begin{tabular}[t]{l}
   #3
\end{tabular}} 
(#2) ;
}

\begin{document}
\title[Sums of two Fibonacci numbers that are perfect powers]{On sums of two Fibonacci numbers that are powers of numbers with limited Hamming weight}
\thanks{The authors were supported by the Austrian Science Fund (FWF) under the project I4406. The first author was 
moreover supported by the Austrian Marshall Plan Foundation with a Marshall Plan Scholarship. She wants to thank Franklin \& Marshall College for their generous hospitality and Eva Goedhart for many helpful conversations.
}


\author[I. Vukusic]{Ingrid Vukusic}
\address{I. Vukusic,
University of Salzburg,
Hellbrunnerstrasse 34/I,
5020 Salzburg, Austria}
\email{ingrid.vukusic\char'100plus.ac.at}

\author[V. Ziegler]{Volker Ziegler}
\address{V. Ziegler,
University of Salzburg,
Hellbrunnerstrasse 34/I,
5020 Salzburg, Austria}
\email{volker.ziegler\char'100plus.ac.at}

\begin{abstract}
In 2018, Luca and Patel conjectured that the largest perfect power representable as the sum of two Fibonacci numbers is  $3864^2 = F_{36} + F_{12}$. In other words, they conjectured that the equation
\begin{equation}\tag{$\ast$}\label{eq:abstract}
	y^a = F_n + F_m
\end{equation}
has no solutions with $a\geq 2$ and $y^a > 3864^2$. While this is still an open problem, there exist several partial results. For example, recently Kebli, Kihel, Larone and Luca proved an explicit upper bound for $y^a$, which depends on the size of $y$. 

In this paper, we find an explicit upper bound for $y^a$, which only depends on the Hamming weight of $y$ with respect to the Zeckendorf representation. 
More specifically, we prove the following: If $y = F_{n_1}+ \dots + F_{n_k}$ and equation~\eqref{eq:abstract} is satisfied by $y$ and some non-negative integers $n,m$ and $a\geq 2$, then 
\[
	y^a 
	\leq \exp\left(C{(\eps)} \cdot k^{(3+\eps)k^2} \right).
\] 
Here, $\eps >0$ can be chosen arbitrarily and $C(\eps)$ is an effectively computable constant.
\end{abstract}

\maketitle

\section{Introduction}\label{sec:intro}

The Fibonacci numbers, defined by $F_0 = 0,$ $F_1 = 1$ and $F_{k+2} = F_{k+1} + F_k$ for $k\geq 0$, might be the most popular linear recurrence sequence of all. 
They have a great many beautiful properties and a vast amount of research has been done on problems involving Fibonacci numbers. 
For instance, it was a long-standing conjecture that 0, 1, 8
and 144 are the only Fibonacci Numbers that are perfect powers. This conjecture was proven in 2003 by Bugeaud, Mignotte and Siksek \cite{BugeaudMignotteSiksek2006}. 
In view of this result, it was a natural next step to search for all perfect powers that are sums of two Fibonacci numbers, i.e.\ to try and solve the equation
\begin{equation}\label{eq:main2fibos}
	F_n + F_m = y^a,
\end{equation}
where $n, m, y, a$ are non-negative integers with $a \geq 2$. 
There are 18 solutions known with $n\geq m \geq 0$, the largest being $F_{36} + F_{12} = 3864^2$.
In 2018, Luca and Patel~\cite{LucaPatel2018} conjectured that these are the only solutions to equation \eqref{eq:main2fibos}. They proved their conjecture in the case that $n \equiv m \pmod{2}$. The general conjecture, however, remains open.

Let us summarize further existing partial results on this conjecture. 
If $m=0$, then we have $F_n = y^a$, which, as mentioned above, was solved in \cite{BugeaudMignotteSiksek2006}. 
For $m=1,2$ we have the equation $F_n + 1 = y^a$, which was solved by Bugeaud, Luca, Mignotte, Siksek in 2006 \cite{BugeaudLucaMignotteSiksek2007}. 
For any fixed $y$ it is in principal possible to solve equation \eqref{eq:main2fibos} completely. For example, 
Bravo and Luca \cite{BravoLuca2016} solved the equation $F_n + F_m = 2^a$. 
In the general case with fixed $y$, explicit upper bounds for $n,m$ and $a$ in terms of $y$ were established recently in \cite{KebliKihelLaroneLuca2021} and in~\cite{KihelLarone2021}. Moreover, Kebli, Kihel, Larone and Luca \cite{KebliKihelLaroneLuca2021} proved that the $abc$-conjecture implies
that \eqref{eq:main2fibos} has only finitely many solutions.
Most recently, Ziegler \cite{Ziegler2022} proved that
for any fixed $y$ equation \eqref{eq:main2fibos} has at most one solution with $a\geq 1$, unless $y = 2, 3, 4, 6, 10$. 
In particular, his result implies that if y can be represented as $y = F_{n_1} + F_{n_2}$, then equation \eqref{eq:main2fibos} has no solutions with $a \geq 2$  (with the exceptions $y = 2, 3, 4, 6, 10$). 

In this paper, we want to make another step towards solving equation~\eqref{eq:main2fibos}, and generalize the above mentioned results in the following way: Instead of fixing $y$ or requiring that it have the form $y = F_{n_1} + F_{n_2}$, we allow arbitrary $y$ with bounded Hamming weight with respect to the Zeckendorf representation (i.e.\ $y = F_{n_1}+ \dots + F_{n_k}$ with bounded $k$). More specifically, we give an explicit upper bound for any perfect power $y^a$ that is a sum of two Fibonacci numbers $y^a = F_n+F_m$, and the upper bound does not depend on the size of $y$, but only on the Hamming weight of the Zeckendorf representation of $y$. We now state our main result.

\begin{thm}\label{thm:main}
Let $\eps >0$. Then there exists an effectively computable constant $C(\eps)$ such that the following holds.
If the equations
\begin{align}
\tag{A}\label{eq:main-y}
	y &= F_{n_1}+ \dots + F_{n_k}
	\quad \text{and}\\
\tag{B}\label{eq:main-ya}
	y^a &= F_n + F_m
\end{align}
are satisfied by some non-negative integers $y, a, n_1 \bb n_k, n, m$ with $a \geq 2$, then
\begin{equation}\label{eq:mainbound}
	y^a 
	\leq \exp\left(C{(\eps)} \cdot k^{(3+\eps)k^2} \right).
\end{equation} 
\end{thm}

\begin{rem}\label{rem:const}
The constant $C(\eps)$ can indeed be computed from our proof. However, it will be extremely large and not useful in practice. This is because we chose to write the upper bound in a way that is both simple and asymptotically good. So if one chooses an $\eps >0$ and 
computes the constant $C(\eps)$ such that \eqref{eq:mainbound} holds for all $k$, the bound will be extremely bad for small $k$.

For computing an actual upper bound for a given $k$, we recommend going to equation~\eqref{eq:Tkplus1} in Section~\ref{sec:walking} and computing the maximum of the expressions $T_{k+1}$ over all $1\leq \ell_0\leq k$. Then, one can proceed as described in Remark~\ref{rem:strategy} in Section~\ref{sec:prelim_fibo} and solve the inequality $n < 6 \cdot 10^{29} \cdot T_{k+1}^4$.
\end{rem}

Let us outline the rest of the paper and the strategy of our proof.
In Section~\ref{sec:prelim}, we state some preliminary results related to Fibonacci numbers and our problem, as well as lower bounds for linear forms in logarithms and an inequality. In Section~\ref{sec:linforms}, we construct a total of $2k$ ``basic'' linear forms in logarithms from equations~\eqref{eq:main-y} and~\eqref{eq:main-ya}. Each of these linear forms will contain the unknown logarithm $\log y$. The main idea of the proof is the following: In several steps ($k$ or $k+1$ steps), we take two of the ``basic'' linear forms at a time and eliminate $\log y$. Then we apply lower bounds for linear forms in logarithms to the new linear form and obtain an upper bound for one of the expressions $n-m, n_1-n_2 \bb n_1-n_k, n_1$. Depending on how large $n-m$ is compared to $n_1-n_2 \bb n_1-n_k, n_1$, we need to do a slightly different succession of steps. An overview of these steps can be found in Figure~\ref{fig:steps} in Section~\ref{sec:elim}. The exact bounds, that are obtained in each step, are computed in Section~\ref{sec:Matveev}. We ``walk the steps'' in Section~\ref{sec:walking}: Depending on which path of steps we walk, we end up with a different bound for $n_1$ in the last step, see Figure~\ref{fig:steps2}. 
In Section ~\ref{sec:walking}, we compute these bounds and find a common upper bound for $n_1$ of the shape $n_1 < c \cdot (\log n)^x$. 
In particular, this implies $\log y < c \cdot (\log n)^x$. 
Finally, in Section~\ref{sec:finish} we combine this bound with the bound from \cite{KebliKihelLaroneLuca2021}, which is of the shape $n<c (\log y)^4$. 
Thus, we end up with an inequality of the shape $n < c \cdot (\log n)^x$, which implies an upper bound for $n$ (see Remark~\ref{rem:strategy}).
Let us moreover point out that we use the result from \cite{LucaPatel2018} to exclude the case $n \equiv m \pmod{2}$. This is very helpful for checking that our linear forms in logarithms don't vanish.

Of course, the strategy of applying lower bounds for linear forms in logarithms to exponential Diophantine equations has been well known and extensively used for a long time. In this paper, we use two particular tricks: the elimination of unknown logarithms and a kind of finite induction. Both of these tricks have been used in several papers before (see e.g.\ \cite{Luca:2017} for the elimination of unknown logarithms and \cite{Luca:2000} for the finite induction method), however, to the authors' best knowledge, this is the first time that they are used in a combined way. Moreover, the induction is not just a straightforward induction over $k$ steps, but different cases lead to quite different bounds. It is interesting to see how the bounds depend on the cases, and to then determine an overall asymptotically good bound.

\section{Preliminary results}\label{sec:prelim}

In this section we, start by recalling some basic properties of Fibonacci numbers. Moreover, we argue why we may assume $n_k \geq 2$ and $n_{i+1}\geq n_{i}+2$ for $i=1 \bb k-1$, as well as $n-2\geq m \geq 2$ in the rest of the paper.
Then we state some known results and some elementary results related to equations \eqref{eq:main-y} and \eqref{eq:main-ya}.

In the second subsection, we state one of Matveev's lower bounds for linear forms in logarithms. Furthermore, we prove an elementary lemma that will allow us to deduce an absolute bound for $n$ from a bound of the shape $n \leq c \cdot (\log n)^x$.

\subsection{Results related to Fibonacci numbers}\label{sec:prelim_fibo}

For the Fibonacci numbers we have the well known Binet formula
\[
	F_n = \frac{\alpha^n - \beta^n}{\sqrt{5}},
	\quad \text{where} \quad
	\alpha = \frac{1+\sqrt{5}}{2}
	\quad \text{and} \quad
	\beta = - 1/\alpha = \frac{1-\sqrt{5}}{2}.
\]

In Theorem~\ref{thm:main} the representations $F_{n_1}+ \dots + F_{n_k}$ and $F_n + F_m$ are not necessarily Zeckendorf representations, i.e.\ we might have consecutive or identical indices, or indices equal to 0 or 1. However, in the rest of this paper we will assume that $n_k \geq 2$ and $n_{i+1}\geq n_{i}+2$ for $i=1 \bb k-1$, as well as $n-2\geq m \geq 2$. Let us justify now why we can do this. 
If $F_{n_1}+ \dots + F_{n_k}$ is not a Zeckendorf representation, then we can consider the Zeckendorf representation $F_{n'_1}+ \dots + F_{n'_{k'}}$ instead. Since the Zeckendorf representation is minimal (see e.g.\ \cite[Theorem 1.1]{Cordwell:2018}) we have $k'\leq k$ and we will get the same or an even stronger result. If $F_n + F_m$ is not a Zeckendorf representation, then either $n=m$, or the Zeckendorf representation in fact only consists of one Fibonacci number and we have $y^a = F_{n'}$. As mentioned in the \nameref{sec:intro}, the latter is by \cite{BugeaudMignotteSiksek2006} only possible for $y^a \leq 144$, in which case we are done.
If $n = m$, then $y^a = 2 F_n$ implies $F_n = 2^s (y')^a$, for suitable $y'$ and $s$. From \cite[Theorem 4]{BugeaudLucaMignotteSiksek2007} it follows that this is only possible for $n=1,2,3,6,12$ (cf.\ \cite[Theorem 2]{LucaPatel2018}). In fact, $y^a = 2F_n$ only works for $n = 3, 6$ and thus $y^a \leq 2F_6 = 16$, and we are done as well.

Finally, note that we may assume $y\geq 2$, since Theorem~\ref{thm:main} is trivial for $y=0,1$.

Next, we state the result due to Luca and Patel \cite{LucaPatel2018}, that we mentioned in the \nameref{sec:intro}. We will use it to exclude the case $n \equiv m \pmod{2}$.

\begin{knownthm}[Luca and Patel, 2018]\label{thm:LucaPatel}
Let $(y, a, n,m)$ be a solution to \eqref{eq:main-ya}, i.e.
\[
	y^a = F_n + F_m,
\]
with $y\geq 1$, $a\geq 2$ and $n-2 \geq m \geq 2$. If $n \equiv m \pmod{2}$, then $n\leq 36$.
\end{knownthm}

Given integers $y,a,n>m$ that satisfy equation~\eqref{eq:main-ya}, one can use lower bounds for linear forms in logarithms to obtain a bound for $n$ in terms of $y$. This was done explicitly by Kebli et al. \cite[Theorem 1]{KebliKihelLaroneLuca2021} and we will use their bound in our proof.

\begin{knownthm}[Kebli et al., 2021]\label{thm:Kebli}
Let $(y, a, n,m)$ be a solution to \eqref{eq:main-ya}, i.e.
\[
	y^a = F_n + F_m,
\]
 with $y\geq 2$, $a\geq 2$ and $n\geq m \geq 0$. Then
\[
	a < n < 6 \cdot 10^{29}(\log y)^4.
\]
\end{knownthm}

Next, we consider equation~\eqref{eq:main-y} and observe the following.

\begin{lem}\label{lem:logyn1}
Let $(y, n_1 \bb n_k)$ be a solution to \eqref{eq:main-y}, i.e.
\[
	y = F_{n_1}+ \dots + F_{n_k},
\]
with $n_k \geq 2$ and $n_{i+1}\geq n_{i}+2$ for $i=1 \bb k-1$.
Then
\[
	\log y < n_1.
\] 
\end{lem}
\begin{proof}
By the properties of the Zeckendorf expansion, we have $y= F_{n_1} + \dots + F_{n_k} < 2 F_{n_1}$, which implies $\log y < \log (2 F_{n_1}) = \log 2 + \log (\alpha^{n_1} - \beta^{n_1}) - \log \sqrt{5} < n_1$.
\end{proof}

\begin{rem}\label{rem:strategy}
In view of Theorem \ref{thm:Kebli} and Lemma \ref{lem:logyn1}, we are done if we manage to prove a bound for $n_1$ of the shape $n_1 \ll (\log n)^x$: Then we have $n < 6 \cdot 10^{29} (\log y)^4 \leq 6 \cdot 10^{29} n_1^4 \ll (\log n)^{4x}$, which immediately gives us an effective upper bound for $n$. This will indeed be our strategy.
\end{rem}

%

We will use the following not very sharp but simple estimates. Note that there is no benefit from making them sharper, as the constants coming from these estimates will be ``swallowed'' by a much larger constant in Section~\ref{sec:Matveev}.

\begin{lem}\label{lem:L}
	For any $k \in \Z_{\geq 1}$ and $n_1\bb,n_k\in\Z$ with $n_1> \dots > n_k \geq 0$, we have
\begin{align}
\label{eq:est_sum_alpha}
	|\alpha^{-n_1} + \dots + \alpha^{-n_k}| &< 3 \quad \text{and} \\
\label{eq:est_sum_beta}
	|\beta^{n_1} + \dots + \beta^{n_k}| &< 3.
\end{align} 
\end{lem}
\begin{proof}
Inequality \eqref{eq:est_sum_alpha} follows from a simple estimation:
\begin{align*}
	0 \leq \alpha^{-n_1} + \dots + \alpha^{-n_k} 
	\leq \sum_{i=0}^{\infty} \alpha^{-i}
	= \frac{1}{1-\alpha^{-1}}
	<3.
\end{align*}
Then \eqref{eq:est_sum_beta} follows easily as well:
\[
	|\beta^{n_1} + \dots + \beta^{n_k}|
	\leq |\beta|^{n_1} + \dots + |\beta|^{n_k}
	= \alpha^{-n_1} + \dots + \alpha^{-n_k} 
	< 3.
\]
\end{proof}

The next lemma will be used to bound the coefficients in the linear forms in Section~\ref{sec:Matveev}. For simplicity, the estimates are very rough. Sharper estimates would improve the bound \eqref{eq:bound_after_steps} at the end of Section~\ref{sec:walking} only marginally. 

\begin{lem}\label{lem:n1n-an}
Let $(y, a, n_1 \bb n_k, n, m)$ be a solution to \eqref{eq:main-y} and \eqref{eq:main-ya} with $y, a, n_k \geq 2$ and $n_{i+1}\geq n_{i}+2$ for $i=1 \bb k-1$, as well as $n-2\geq m \geq 2$. Then we have
\[
	n_1 < n 
	\quad \text{and} \quad
	a < n.
\]	
\end{lem}
\begin{proof}
Since $y^a > y$, it follows immediately from the properties of the Zeckendorf representation, that $n_1 < n$. 

For the second inequality note that we have $2^a \leq y^a = F_n + F_m < 2 F_n < 2 \alpha^n$, which implies $a < \log 2 + n \cdot \log \alpha / \log 2 <n$ for $n\geq m+2 \geq 4$.

\end{proof}

\subsection{Result related to the application of lower bounds for linear forms in logarithms}

Our proof will heavily rely on lower bounds for linear forms in logarithms. In order to switch from expressions of the shape $|\eta_1^{b_1} \cdots \eta_t^{b_t} - 1|$ to linear forms in logarithms of the shape $|b_1 \log \eta_1 + \dots + b_t \log \eta_t|$, we will use the following easy-to-check lemma.

\begin{lem}\label{lem:standard-linform}
	If $|x-1|\leq 0.5$, then $|\log x| \leq 2 |x-1|$.
\end{lem}

In order to compute lower bounds for expressions of the shape $|b_1 \log \eta_1 + \dots + b_t \log \eta_t|$, we will use Matveev's popular result \cite[Corollary 2.3]{Matveev2000} because it is very good and easy to apply.

Let us first recall the definition of the height and some basic properties.
Let $ \eta $ be an algebraic number of degree $ d $ over the rationals, with minimal polynomial
	\begin{equation*}
		a_0 (X - \eta^{(1)}) \cdots (X - \eta^{(d)}) \in \Z[X].
	\end{equation*}
	Then the absolute logarithmic height of $ \eta $ is given by
	\begin{equation*}
		h(\eta) = \frac{1}{d} \left( \log a_0 + \sum_{i=1}^{d} \log \max \{1, |\eta^{(i)}|\} \right).
	\end{equation*}
For any algebraic numbers $ \eta_1 \bb \eta_t \in \overline{\Q} $ and $ z \in \Z $, the following well known properties hold:
	\begin{align*}
		h(\eta_1 + \dots + \eta_t) &\leq h(\eta_1) + \dots + h(\eta_t) + \log t, \\
		h(\eta^z) &= |z| \cdot  h(\eta).
	\end{align*}

\begin{knownthm}[Matveev, 2000]\label{thm:Matveev}
Let $\eta_1, \ldots, \eta_t$ be positive real algebraic numbers in a number field $K$ of degree $D$, let $b_1, \ldots b_t$ be rational integers and assume that
\[
	\Lambda := b_1 \log \eta_1 + \dots + b_t \log \eta_t \neq 0.
\]
Then 
\[
	\log |\Lambda|
	\geq -1.4 \cdot 30^{t+3} \cdot t^{4.5} \cdot D^2 (1 + \log D) (1 + \log B) A_1 \cdots A_t,
\]
where
\begin{align*}
	B &\geq \max \{|b_1|, \ldots, |b_t|\},\\
	A_i &\geq \max \{D h(\eta_i), |\log \eta_i|, 0.16\}
	\quad \text{for all }i = 1,\ldots,t.
\end{align*}
\end{knownthm}

In order to be able to apply Theorem~\ref{thm:Matveev}, one has to check that the linear form $\Lambda$ does not vanish. This often requires some tricks. In Section~\ref{sec:Matveev} we will make use of the following lemma.

\begin{lem}\label{lem:teilbarkeitSqrt5}
Let $x$ be an odd positive integer. Then $\alpha^x +1$ is not divisible by $\sqrt{5}$ (in the principal ideal domain $\Z[\alpha]$, where $\alpha =\frac{1+\sqrt{5}}{2}$).
\end{lem}
\begin{proof}
Let $x$ be an odd positive integer.
If $\alpha^x +1$ were divisible by $\sqrt{5}$, then the norm $N_{\Q(\sqrt{5})/\Q}(\alpha^x+1)$ would be divisible by 5. Let us compute the norm:
\begin{align*}
	N_{\Q(\sqrt{5})/\Q}(\alpha^x+1)
	= (\alpha^x+1)(\beta^x+1)
	= (\alpha \beta)^x + (\alpha^x + \beta^x) + 1
	= (-1)^x + L_x + 1
	= L_x,
\end{align*}
where $L_x$ is the $x$-th Lucas number (the Lucas numbers are defined by $L_0=2$, $L_1=1$ and $L_n= L_{n-1} + L_{n-2}$ for $n\geq 2$, and they indeed have the Binet representation $L_n=\alpha^n + \beta^n$). Modulo 5 the Lucas sequence looks like this: 
$2,1,3,4,2,1,\ldots$
and in particular no Lucas number is divisible by 5. Thus  $N_{\Q(\sqrt{5})/\Q}(\alpha^x+1)$ is not divisible by 5 for odd $x$, and $\alpha^x+1$ cannot be divisible by $\sqrt{5}$.
\end{proof}

Finally, after the repeated application of lower bounds for linear forms in logarithms, we will end up with an inequality of the shape $n \leq c \cdot (\log n)^x$ by the end of Section~\ref{sec:walking}. In order to obtain an absolute upper bound for $n$, we will use the following lemma in Section~\ref{sec:finish}.

\begin{lem}\label{lem:ungl}
Let $\delta>0$ and let $n, c, x  \in \R_{\geq 1}$ satisfy the inequality
\begin{equation}\label{eq:lem:ungl}
	n \leq c \cdot (\log n)^x.
\end{equation}
Then
\[
	n \leq \max \left\{ 
		\exp (\exp ( ( 1+ \delta^{-1})^2) ),\		
		2^x \cdot c \cdot (\log c)^x,\
		(2x)^{(1+\delta)x} \cdot c
		\right\}.
\]
\end{lem}
\begin{proof}
Let $n,c,x$ be numbers $\geq 1$ that satisfy \eqref{eq:lem:ungl}.

First, note that if $n \leq e^e$, then we have $n < \exp (\exp ( ( 1+ \delta^{-1})^2) )$ and we are done immediately. Therefore, we may assume that $n > e^e$, which implies
\begin{equation}\label{eq:nee}
	\log \log \log n > 0.
\end{equation}

Next, assume for a moment that $\delta / (1+ \delta) \cdot \log \log n \leq \log \log \log n$. Taking logarithms, we obtain
\[
	\log (\delta / (1+ \delta)) + \log \log \log n
	\leq \log \log \log \log n.
\]
Since $\log y < y/2$ for any $y>0$, and $\log \log \log n > 0$ by \eqref{eq:nee}, we obtain
\[
	\log (\delta / (1+ \delta)) + \log \log \log n
	\leq \log \log \log \log n
	< (\log \log \log n)/2,
\]
which implies
\[
	\log \log \log n
	< 2 \log ((1+ \delta)/\delta) 
	= 2 \log (1 + \delta^{-1}).
\]
But this immediately implies $n < \exp (\exp ( ( 1+ \delta^{-1})^2) )$ and we are done. Therefore, we may from now on assume that
\begin{equation}\label{eq:nnotlastbound}
	\log \log \log n  
	< \delta / (1+ \delta) \cdot \log \log n.
\end{equation}

Now we consider inequality \eqref{eq:lem:ungl} and take logarithms, obtaining
\begin{equation}\label{eq:lem:ungl-log}
	\log n  \leq \log c + x \log \log n.
\end{equation}

Assume for the moment that $x \log \log n \leq (\log n)/{2}$. Then the above inequality implies
\begin{equation*}
	\log n \leq 2 \log c,
\end{equation*}
which plugging back into \eqref{eq:lem:ungl} immediately yields $n\leq 2^x \cdot c \cdot (\log c)^x$ and we are done as well.

Finally, we consider the case that $x \log \log n > (\log n)/{2}$, i.e.\ $\log n < 2 x \log \log n$. Taking logarithms, we obtain
\[
	\log \log n 
	< \log (2x) + \log \log \log n.
\]
Together with \eqref{eq:nnotlastbound} this yields
\[
	\log \log n 
	< \log (2x) + \delta / (1+ \delta) \cdot \log \log n,
\]
which implies
\[
	\left( 1 - \frac{\delta}{1+\delta} \right) \log \log n 
	= \frac{1}{1+\delta} \cdot \log \log n 
	< \log (2x),
\]
and then
\[
	\log \log n 
	< (1+\delta) \log (2x).
\]
Plugging this into \eqref{eq:lem:ungl-log}, we obtain
\[
	\log n 
	< \log c + x \cdot (1+\delta) \log (2x),
\]
which implies
\[
	n
	< c \cdot (2x)^{(1+\delta)x}.
\]
\end{proof}

\section{Constructing the basic linear forms in logarithms}\label{sec:linforms}

In this section, we construct $k$ linear forms in logarithms from equation~\eqref{eq:main-y} and $k$ linear forms from equation~\eqref{eq:main-ya}. In the third subsection we give an overview of all the linear forms in logarithms and their upper bounds.

\subsection{Linear forms coming from equation \eqref{eq:main-y}}
Using the Binet formula, we can rewrite equation~\eqref{eq:main-y} as
\begin{equation*}\label{eq:main-y-binet}
	y = \frac{\alpha^{n_1} - \beta^{n_1}}{\sqrt{5}} + \dots + \frac{\alpha^{n_k} - \beta^{n_k}}{\sqrt{5}}.
\end{equation*}
Multiplication by $\sqrt{5}$ yields
\begin{equation}\label{eq:main-y-umgef}
	y\sqrt{5} = \alpha^{n_1} + \dots + \alpha^{n_k} - (\beta^{n_1} + \dots + \beta^{n_k}).
\end{equation}
To obtain the first linear form in logarithms, we shift the largest power $\alpha^{n_1}$ to the left hand side, take absolute values and use Lemma~\ref{lem:L}:
\begin{align*}
	| y \sqrt{5} - \alpha^{n_1} | 
	&= |\alpha^{n_2} + \dots + \alpha^{n_k} - (\beta^{n_1} + \dots + \beta^{n_k})| \\
	&\leq |\alpha^{n_2} ( \alpha^{n_3 - n_2} + \dots + \alpha^{n_k - n_2} )| + |\beta^{n_1} + \dots + \beta^{n_k}|\\
	&\leq \alpha^{n_2} \cdot 3 + 3\\
	&\leq 6 \alpha^{n_2}.
\end{align*}
Now we divide by $\alpha^{n_1}$ and obtain
\[
	\left| \frac{y \sqrt{5}}{\alpha^{n_1}} - 1 \right|
	\leq 6 \alpha^{-(n_1-n_2)}.
\]
If $n_1-n_2\geq 6$, then the above expressions are $\leq 0.5$, so
with Lemma \ref{lem:standard-linform} we get
\begin{equation*}\label{eq:linform1}
	| \log y + \sqrt{5} - n_1 \log \alpha |
	\leq 12 \alpha^{-(n_1-n_2)}.
\end{equation*}
This is our first linear form in logarithms. Next, we construct more of them by shifting more powers of $\alpha$ to the left hand side.

Let us go back to \eqref{eq:main-y-umgef} and shift the largest $\ell$ powers $\alpha^{n_1} \bb \alpha^{n_\ell}$ to the left ($2\leq \ell \leq k-1$). Then, as above, we obtain
\begin{align*}
	| y\sqrt{5} - (\alpha^{n_1} + \dots + \alpha^{n_\ell}) |
	= |\alpha^{n_{\ell+1}} + \dots + \alpha^{n_k} - (\beta^{n_1} + \dots + \beta^{n_k})|
	\leq 6 \alpha^{n_{\ell+1}}.
\end{align*}
Dividing by $\alpha^{n_1} + \dots + \alpha^{n_\ell} = \alpha^{n_1}(1 + \alpha^{n_2-n_1} + \dots + \alpha^{n_\ell - n_1}) \geq \alpha^{n_1}$ we obtain
\begin{equation*}
		\left| 
			\frac{y \sqrt{5}}{\alpha^{n_1}(1 + \alpha^{n_2-n_1} + \dots + \alpha^{n_\ell - n_1})}
		\right| 
		\leq 6 \alpha^{-(n_1 - n_{\ell + 1})}.
\end{equation*}
Then, if $n_1 - n_{\ell+1} \geq 6$, by Lemma \ref{lem:standard-linform}  we have
\begin{equation*}
	|\log y + \log \sqrt{5} - n_1 \log \alpha - \log (1 + \alpha^{n_2-n_1} + \dots + \alpha^{n_\ell - n_1})|
	\leq 12\alpha^{-(n_1 - n_{\ell + 1})}.
\end{equation*}

Finally, let us construct the last linear form in logarithms by shifting all powers of $\alpha$ to the left hand side. Then from equation \eqref{eq:main-y-umgef} we obtain
\begin{align*}
	| y\sqrt{5} - (\alpha^{n_1} + \dots + \alpha^{n_k}) |
	= |\beta^{n_1} + \dots + \beta^{n_k}|
	\leq 3
	\leq 6.
\end{align*}
Dividing by $\alpha^{n_1} + \dots + \alpha^{n_k} = \alpha^{n_1} (1 + \alpha^{n_2-n_1} + \dots + \alpha^{n_k - n_1}) \geq \alpha^{n_1}$, we obtain
\begin{equation*}
		\left| 
			\frac{y \sqrt{5}}{\alpha^{n_1}(1 + \alpha^{n_2-n_1} + \dots + \alpha^{n_k - n_1})}
		\right| 
		\leq 6 \alpha^{-n_1}.
\end{equation*}
Then, as above, if $n_1 \geq 6$, with Lemma \ref{lem:standard-linform}  we obtain
\begin{equation*}
	|\log y + \log \sqrt{5} - n_1 \log \alpha - \log (1 + \alpha^{n_2-n_1} + \dots + \alpha^{n_k - n_1})|
	\leq 12\alpha^{-n_1}.
\end{equation*}

\subsection{Linear forms coming from Equation \eqref{eq:main-ya}}

We rewrite equation \eqref{eq:main-ya} using the Binet formula:
\begin{equation*}\label{eq:main-ya-binet}
	y^a = \frac{\alpha^{n} - \beta^{n}}{\sqrt{5}} + \frac{\alpha^{m} - \beta^{m}}{\sqrt{5}}.
\end{equation*}
Then, by the same procedure as above, if $n-m\geq 6$ and $n\geq 6$, we obtain the two linear forms
\begin{align*}
	| a \log y + \sqrt{5} - n \log \alpha |
	&\leq 12 \alpha^{-(n-m)}
	\quad \text{and}\\
	|a \log y + \log \sqrt{5} - m \log \alpha - \log (\alpha^{n-m} + 1)|
	&\leq 12\alpha^{-n}.
\end{align*}

\subsection{Overview of basic linear forms}

Let us sum up all the linear forms in logarithms that we have constructed and name them in an appropriate way: Let us denote the $k$
linear forms coming from equation~\eqref{eq:main-y} by $\Lambda_{A1} \bb \Lambda_{Ak}$, and the two linear forms coming from equation~\eqref{eq:main-ya} by $\Lambda_{B1}, \Lambda_{B2}$. Specifically, for $2\leq l \leq k-1$, we have
\begin{align*}
|\Lambda_{A1}|:=	
	| \log y + \log \sqrt{5} - n_1 \log \alpha | 
		&\leq 12 \alpha^{-(n_1-n_2)},\\
|\Lambda_{A\ell}|:=		
	|\log y + \log \sqrt{5} - n_1 \log \alpha - \log (1 + \alpha^{n_2-n_1} + \dots + \alpha^{n_\ell - n_1})|
		&\leq 12\alpha^{-(n_1 - n_{\ell + 1})}, \\
|\Lambda_{Ak}|:=	
	|\log y + \log \sqrt{5} - n_1 \log \alpha - \log (1 + \alpha^{n_2-n_1} + \dots + \alpha^{n_k - n_1})|
		&\leq 12\alpha^{-n_1},\\
|\Lambda_{B1}|:=
	| a \log y + \log \sqrt{5} - n \log \alpha |
		&\leq 12 \alpha^{-(n-m)},\\
|\Lambda_{B2}|:=	
	|a \log y + \log \sqrt{5} - m \log \alpha - \log (\alpha^{n-m} + 1)|
	&\leq 12\alpha^{-n}.
\end{align*}
The upper bounds are all of the shape $12 \alpha^{-X}$.
Mind that each of the inequalities only holds if $X \geq 6$. However, the goal will always be to bound $X$ from above, so if $X\leq 6$, we will just skip the corresponding step.

\section{Eliminating $y$ from the linear forms and overview of the steps}\label{sec:elim}

In each of the linear forms $\Lambda_{A1} \bb \Lambda_{Ak}, \Lambda_{B1}, \Lambda_{B2}$ we have one logarithm that is unknown, namely $\log y$. Therefore, we will always take two distinct linear forms and eliminate $\log y$. For example, if we take  $\Lambda_{A1}$ and $\Lambda_{B1}$ and eliminate $\log y$, we get a new linear form $\Lambda_{A1}^*$ in fixed logarithms with a bound of the shape $|\Lambda_{A1}^*| \leq a \cdot 12 \alpha^{-(n_1-n_2)} + 12 \alpha^{-(n-m)}$. After computing a lower bound for $|\Lambda_{A1}^*|$ with Matveev's theorem, we then obtain an upper bound either for $n_1-n_2$ or for $n-m$, depending on which one is smaller. This leads to different cases and in each case we have to continue with slightly different steps. Figure~\ref{fig:steps} shows an overview of the steps. Each rectangular box stands for a step, where we take the two linear forms written in the box and eliminate $y$. The obtained upper bound is written along the arrow that points to the next step.
In the steps on the left (Steps A1 to A$k$), there are always two cases, depending on whether $n_1-n_\ell$ (or $n_1$) or $n-m$ is smaller. If we ever cross over to the steps on the right (Steps B1 to B$k$), then we just follow the arrows pointing downwards, always obtaining bounds for $n_1-n_\ell$ (or $n_1$).
The purpose of Figure~\ref{fig:steps} is to give a rough idea of the proof. There will be a more detailed figure in Section~\ref{sec:walking}.
We now construct and estimate all the linear forms in logarithms that are used in the steps.

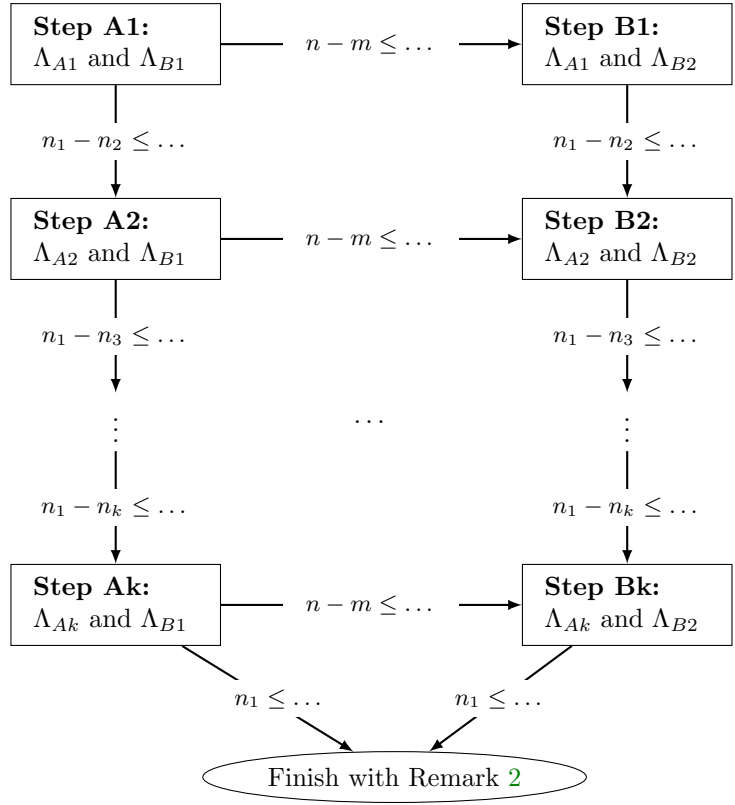
\begin{figure}[h]
\begin{tikzpicture}[auto,node distance=1.5cm]

\mybox{A1}{}
{$\Lambda_{A1}$ and $\Lambda_{B1}$}

\mybox{A2}{below =of stepA1}
{$\Lambda_{A2}$ and $\Lambda_{B1}$}

\node[](dotsleft)[below = of stepA2] {\vdots};

\mybox{Ak}{below =of dotsleft}
{$\Lambda_{Ak}$ and $\Lambda_{B1}$}

\mybox{B1}{right = 4cm of stepA1}
{$\Lambda_{A1}$ and $\Lambda_{B2}$}

\mybox{B2}{right = 4cm of stepA2}
{$\Lambda_{A2}$ and $\Lambda_{B2}$}

\node[](dotsright)[below = of stepB2] {\vdots};

\mybox{Bk}{right = 4cm of stepAk}
{$\Lambda_{Ak}$ and $\Lambda_{B2}$}

\node (dotsmiddle) at ($(dotsleft)!0.5!(dotsright)$) {\dots};

\node[ellipse, draw] (doneEnd) [below right = 1.5cm and 0.5cm of stepAk] 
{Finish with Remark \ref{rem:strategy}};


\pfeilrueber{stepA1}{stepB1}{$n-m \leq \dots$}
\pfeilrueber{stepA2}{stepB2}{$n-m \leq \dots$}
\pfeilrueber{stepAk}{stepBk}{$n-m \leq \dots$}

\pfeilrunter{stepA1}{stepA2}{$n_1-n_2$ $\leq \dots$}
\pfeilrunter{stepA2}{dotsleft}{$n_1-n_3$ $\leq \dots$}
\pfeilrunter{dotsleft}{stepAk}{$n_1-n_{k}$ $\leq \dots$}

\pfeilrunter{stepB1}{stepB2}{$n_1-n_2$ $\leq \dots$}
\pfeilrunter{stepB2}{dotsright}{$n_1-n_3$ $\leq \dots$}
\pfeilrunter{dotsright}{stepBk}{$n_1-n_k$ $\leq \dots$}

\draw [-latex,thick] (stepAk) --  node [fill=white, anchor=center, pos=0.5] {\small{\quad $n_1 \leq \dots$}} (doneEnd) ;
\pfeilrunter{stepBk}{doneEnd}{$n_1 \leq \dots$}

\end{tikzpicture}
\caption{Overview of steps}
\label{fig:steps}
\end{figure}

Let us start with the steps on the left (Step A1 to Step A$k$).
In Step A1 we eliminate $\log y$ by computing the new linear form $\Lambda_{A1}^* := a \Lambda_{A1} - \Lambda_{B1}$:
\begin{align*}
	|\Lambda_{A1}^*|
	= |a \Lambda_{A1} - \Lambda_{B1}|
	&= | (a-1) \log \sqrt{5} + (n - a n_1) \log \alpha| \\
	&\leq a \cdot 12 \alpha^{-(n_1-n_2)} + 12 \alpha^{-(n-m)}
	\leq 18 a \alpha^{-\min\{n_1-n_2,n-m\}}, 
\end{align*}
where we used $a\geq 2$.
In the same way we also obtain the linear forms for Steps A2 to A$k$:\
\begin{align*}
	| \Lambda_{A\ell}^*|
	:= | a \Lambda_{A\ell} - \Lambda_{B1} |
	&= | (a-1) \log \sqrt{5} + (n - a n_1) \log \alpha - a \log (1 + \alpha^{n_2-n_1} + \dots + \alpha^{n_\ell - n_1}) | \\
	&\leq 18 a \alpha^{-\min\{n_1-n_{\ell+1},n-m\}}
	\quad \text{for } l= 2 \bb k-1,\\
	|\Lambda_{Ak}^*| 
	:= | a \Lambda_{Ak} - \Lambda_{B1}|
	& = | (a-1) \log \sqrt{5} + (n - a n_1) \log \alpha - a \log (1 + \alpha^{n_2-n_1} + \dots + \alpha^{n_k - n_1}) | \\
	& \leq 18 a \alpha^{-\min\{n_1,n-m\}}.\nonumber
\end{align*}

Analogously, we construct and estimate the linear forms for Steps B1 to B$k$. Note that $n_1-n_\ell<n_1 <n$, so clearly $\alpha^{-n} < \alpha^{-(n_1-n_\ell)}$ for any $\ell$ and we don't need to write any minima.
 
\begin{align}
|\Lambda_{B1}^*| \nonumber 
	:= |a \Lambda_{A1} - \Lambda_{B2}|
	&= |(a-1)\log \sqrt{5} + (m-an_1)\log \alpha + \log(\alpha^{n-m} +1)|\\
	&\leq a \cdot 12\alpha^{-(n_1-n_2)} + 12 \alpha^{-n}
	\leq 18a \alpha^{-(n_1-n_2)}, \nonumber \\
|\Lambda_{B\ell}^*| \nonumber 
	:= |a \Lambda_{A\ell} - \Lambda_{B2}|
	&= |(a-1)\log \sqrt{5} + (m-an_1)\log \alpha  \\
		& \phantom{= |} - a \log (1 + \alpha^{n_2-n_1} + \dots + \alpha^{n_\ell - n_1}) + \log(\alpha^{n-m} +1)| \nonumber \\
	&\leq 18a \alpha^{-(n_1-n_{\ell+1})},
	\quad \text{for } l= 2 \bb k-1, \nonumber \\
|\Lambda_{Bk}^*| \nonumber 
	:= |a \Lambda_{Ak} - \Lambda_{B2}|
	&= |(a-1)\log \sqrt{5} + (m-an_1)\log \alpha  \\
		& \phantom{= |} - a \log (1 + \alpha^{n_2-n_1} + \dots + \alpha^{n_k - n_1}) + \log(\alpha^{n-m} +1)| \nonumber \\
	&\leq 18a \alpha^{-n_1}. \label{eq:Lambdabkstar}
\end{align}

Mind that each of these estimates only holds if we have $X\geq 6$ for the exponent in the upper bound $18 a \alpha^{-X}$.  

\section{Application of Matveev's theorem}\label{sec:Matveev}

In this section we apply Matveev's theorem to all the linear forms  $\Lambda_{A\ell}^*, \Lambda_{B\ell}^*$, that we obtained in the previous section. Moreover, we compare the lower bounds to the upper bounds and compute general bounds for the exponents $n_1 - n_\ell$, $n_1$ or $n-m$.

First, we check that the linear forms $\Lambda_{A1}^*\bb \Lambda_{Ak}^*, \Lambda_{B1}^* \bb  \Lambda_{Bk}^*$ are non-zero. Note that in each linear form we have the expression $(a-1)\log \sqrt{5}$ and $a-1\neq 0$. Thus, if a linear form were zero, $\log \sqrt{5}$ would have to be canceled out by the other logarithms. In particular, since $\sqrt{5}$ is prime in $\Z[\alpha]$, it would have to divide at least one argument of a logarithm.
We only have three other types of logarithms: $\log \alpha$, $\log (\alpha^{n-m}+1)$ and $\log (1 + \alpha^{n_2-n_1} + \dots + \alpha^{n_\ell - n_1}) $.
First, since $\alpha$ is a unit in $\Z[\alpha]$, it is not divisible by $\sqrt{5}$, so $\log \alpha$ does not contribute to the cancellation of $\log \sqrt{5}$. 
Second, by Lemma~\ref{lem:teilbarkeitSqrt5}, the expression $\alpha^{n-m}+1$ is not divisible by $\sqrt{5}$ unless $n-m$ is even. 
If $n-m$ is even, we are immediately done with Theorem~\ref{thm:LucaPatel}, so let us assume that $n-m$ is odd. Thus $\log (\alpha^{n-m}+1)$ does not contribute to the cancellation of $\log \sqrt{5}$ either.
And third, the expression $(1 + \alpha^{n_2-n_1} + \dots + \alpha^{n_\ell - n_1})$ might be divisible by $\sqrt{5}$. However, the coefficient of $\log (1 + \alpha^{n_2-n_1} + \dots + \alpha^{n_\ell - n_1}) $ is always $-a$. So if $\log \sqrt{5}$ canceled out completely, we would need to have $(a-1) - a \cdot v_{\sqrt{5}}(1 + \alpha^{n_2-n_1} + \dots + \alpha^{n_\ell - n_1})=0$, which is clearly impossible. Here $v_{\sqrt{5}}(\cdot)$ denotes the valuation on $\Q(\sqrt{5})$ associated with the prime ideal $(\sqrt{5})$, normalized by $v_{\sqrt{5}}(\sqrt{5})=1$.

Therefore, all the linear forms $\Lambda_{A1}^*\bb \Lambda_{Ak}^*, \Lambda_{B1}^* \bb  \Lambda_{Bk}^*$ are non-zero and we can apply Matveev's  theorem to each of them. Each linear form has an upper bound of the shape $18 a \alpha^{-X}$ and in each step we compare the lower and the upper bound to obtain a bound for the expression $X$.

We start by describing the last step, because the lower bound for the $\Lambda_{Bk}^*$ will be the weakest, so for simplicity we will be able to reuse it in the other steps. 
Of course, one can obtain sharper bounds by considering each linear form separately, in particular for the linear form $\Lambda_{A1}^*$ in only two logarithms.

\subsection{Step Bk}

Assume that we already have an upper bound 
$T_{k}$ for $n_1-n_k$ and an upper bound $S$ for $n-m$. 
We want to apply Matveev's theorem (Theorem~\ref{thm:Matveev}) to $\Lambda_{Bk}^*$ with 
$t=4$ and
\begin{align*}
	\eta_1 &= \sqrt{5}, \quad && b_1 = a-1,\\
	\eta_2 &= \alpha , \quad && b_2 = m-an_1,\\
	\eta_3 &= 1 + \alpha^{n_2-n_1} + \dots + \alpha^{n_k - n_1}, \quad && b_3 = -a,\\
	\eta_4 &= \alpha^{n-m} + 1, \quad && b_4 = 1.
\end{align*}
The four numbers $\eta_1, \eta_2, \eta_3, \eta_4$ are real, positive and belong to 
$K=\Q(\sqrt{5})$, so we can take $D=2$. As stated in Lemma~\ref{lem:n1n-an}, we have $a\leq n$ and $n_1\leq n$, so $\max\{|b_1|, |b_2|,|b_3|, |b_4|\}\leq n^2$ and we can set $B=n^2$. Since $h(\sqrt{5})=(\log 5)/2$ and $h(\alpha)= (\log \alpha)/2$, we can set $A_1= \log 5$ and $A_2=\log \alpha$. Finally, we estimate (rather roughly) the heights of $\eta_3$ and $\eta_4$. Recall that we are assuming $n_1-n_k\leq T_k$ and $n-m \leq S$.
\begin{align*}
h(1 + \alpha^{n_2-n_1} + \dots + \alpha^{n_k - n_1})
	&\leq h(1) + h(\alpha^{n_2-n_1}) + \dots + h(\alpha^{n_k - n_1}) + \log k\\
	&= 0 + (n_1-n_2) h(\alpha) + \dots + (n_1 - n_k) h(\alpha) + \log k \\
	&\leq (k-1) (n_1-n_k) \frac{\log \alpha}{2} + \log k\\
	&\leq k T_k.\\
h(\alpha^{n-m}+1)
	&\leq (n-m) \frac{\log \alpha}{2} + \log 2
	\leq S.
\end{align*}

Thus we can set $A_3=2kT_k$ and $A_4=2S$. Now 
Matveev's theorem (Theorem~\ref{thm:Matveev}) tells us that
\begin{equation*}\label{eq:matvtoLambdak+k}
	\log |\Lambda_{Bk}^*|
		\geq - 1.4 \cdot 30^{4+3} \cdot 4^{4.5} \cdot 2^2 (1+\log 2) (1+ \log (n^2)) \log 5 \cdot \log \alpha \cdot 2kT_k \cdot 2S.
\end{equation*}
We can estimate $1+ \log n^2 = 1 + 2\log n \leq 3 \log n$ (this estimate holds for $n \geq 3$; if $n<3$, we are immediately done). Then we simplify the above lower bound to 
\begin{equation}\label{eq:estimatelowerbound-Lambdak+kstar}
	\log |\Lambda_{Bk}^*|
		\geq - \Ck \cdot \log n \cdot \log \alpha \cdot k \cdot T_k \cdot S,
\end{equation}
with
\begin{align*}
	\Ck 
	&:= 2.1 \cdot 10^{15}\\
	&\geq 1.4 \cdot 30^{4+3} \cdot 4^{4.5} \cdot 2^2 (1+\log 2) \cdot 3 \cdot \log 5 \cdot 2 \cdot 2.
\end{align*}

Together with \eqref{eq:Lambdabkstar} this yields
\begin{equation*}\label{eq:comparebounds_k+k}
	- \Ck \cdot \log \alpha \cdot S \cdot T_k \cdot \log n
	\leq \log |\Lambda_{Bk}^*|
	\leq \log 18 + \log a - n_1 \log \alpha,
\end{equation*}
which implies 
\begin{equation}\label{eq:bound_stepBk}
	n_1 \leq \Ck k S T_k \log n
	=: T_{k+1}.
\end{equation}
Note that we are allowed to omit the expressions $\log 18$ and $\log a$ because $\log a \leq \log n$ and the constant $\Ck$ is extremely large and was estimated very roughly.
Moreover, note that the upper bound coming from \eqref{eq:Lambdabkstar} only holds if $n\geq 6$. However, if $n<6$, then \eqref{eq:bound_stepBk} is trivially fulfilled. An analogous argument will implicitly be used in all other steps as well.

\subsection{Steps B1 to Bk--1}

First, assume that $\ell\in \{2\bb k-1\}$ and assume that we have already found bounds $T_\ell \geq n_1-n_\ell$ and $S \geq n-m$. Then in Step B$\ell$ we consider the linear form $\Lambda_{B\ell}^*$ and, completely analogously to Step B$k$, we obtain a lower bound $\log |\Lambda_{B\ell}^*| \geq 
-\Ck \cdot \log n \cdot \log \alpha \cdot \ell \cdot T_\ell \cdot S$. By comparing the upper and lower bound for $|\Lambda_{B\ell}^*|$ we then, analogously to Step B$k$, obtain
\begin{equation}\label{eq:bound_stepBl}
	n_1-n_{\ell +1}
	\leq \Ck \ell S T_\ell \log n
	=:T_{\ell +1}.
\end{equation} 

Now let us have a look at Step B$1$. The linear form $\Lambda_{B1}^*$ has only three logarithms (instead of four), so it is possible to get a better lower bound. However, for simplicity, we can set $T_1 = 1$  and see that a bound analogous to \eqref{eq:estimatelowerbound-Lambdak+kstar} still holds: $\log |\Lambda_{B1}^*| \geq - \Ck \cdot \log n \cdot \log \alpha \cdot 1 \cdot 1 \cdot S$. Thus, in Step B1 we can also obtain the bound \eqref{eq:bound_stepBl} for $\ell=1$ and $T_1=1$.

\subsection{Steps A1 to Ak}
Let $\ell \in \{1\bb k\}$. In Step A$\ell$ we assume that we already have a bound $R_\ell \geq n_1-n_\ell$ (if $\ell = 1$, we have no bound yet). We consider the linear form $\Lambda_{A\ell}^*$, which is almost exactly equal to the corresponding linear form $\Lambda_{B\ell}^*$, except that the logarithm $\log(\alpha^{n-m} + 1)$ is missing. Thus, we only have three logarithms (or even two logarithms if $\ell = 1$) and we can obtain even better lower bounds than in \eqref{eq:estimatelowerbound-Lambdak+kstar}. 
However, for simplicity we set $S=1$ and see that the lower bound
\[
	\log |\Lambda_{A\ell}^*|
	\geq - \Ck \cdot \log n \cdot \log \alpha \cdot \ell \cdot R_\ell
\]
holds (where $R_1=1$ if $\ell =1$).
Thus, by comparing upper and lower bounds, we obtain
\begin{equation}\label{eq:boundStepAl}
	\min \{n_1 - n_{\ell + 1}, n- m \} 
	\leq \Ck \ell R_\ell \log n
\end{equation}
for $l<k$.
If $\min \{n_1 - n_{\ell + 1}, n- m \} = n_1 - n_{\ell + 1}$, 
we set $R_{\ell + 1} := \Ck \ell R_\ell \log n$, otherwise we set $S_{k}:=\Ck \ell R_\ell \log n$. 
For $\ell = k$ we obtain 
\[
	\min \{n_1, n- m \} 
	\leq \Ck k R_k \log n,
\]
and if $\min \{n_1, n- m \} = n_1$, we set $T_{k+1} = \Ck k T_k \log n$; otherwise we set $S_{k} := \Ck k T_k \log n$.

\section{Walking the steps}\label{sec:walking}

Depending on how large $n-m$ is compared to $n_1 - n_2 \bb n_1 - n_k, n_1$, we do a specific series of steps, starting with Step A1. 

Say we are in Step A$\ell$. If $\min\{n_1-n_{\ell+1}, n- m \} = n_1-n_{\ell+1}$, then we get a bound $R_{\ell+1}$ for $n_1-n_{\ell+1}$ and we continue with Step A$\ell+1$. In the other case that $n-m$ is the minimum, we get a bound $n-m \leq S_{\ell}$ and we continue with Step B$\ell$, and
after that, we continue with Steps B$\ell+1$ all the way until Step B$k$. This is illustrated in Figure~\ref{fig:steps2}. Note that the question marks stand for numbers and letters that depend on the case, i.e.\ which path we are coming from. For example, if $n-m \leq n_1 - n_2$, then we do  Step A1 -- Step B1 -- Step B2 -- $\cdots$. In this case, in Step B2 we get $n_1 - n_3 \leq T_3 = \Ck \cdot 2 \cdot S_1 \cdot T_2 \cdot \log n$. On the other hand, if $n_1 - n_2 \leq n-m \leq n_1 - n_3$, then we do Step A1 -- Step A2 -- Step B2 -- $\cdots$ and we get the bound $n_1 - n_3 \leq T_3 = \Ck \cdot 2 \cdot S_2 \cdot R_2 \cdot \log n$.

If $n_1 < n-m$, we finish with Step A$k$, obtaining the bound $n_1\leq R_{k+1}$. In all other cases, we finish with Step B$k$, obtaining a bound $n_1 \leq T_{k+1}$.

\newcommand{\myboxt}[3]{
\node[entity] (step#1)[#2] {
 \begin{tabular}[t]{p{3cm}}
   \textbf{Step #1:} \\
   #3
\end{tabular}
 };
}

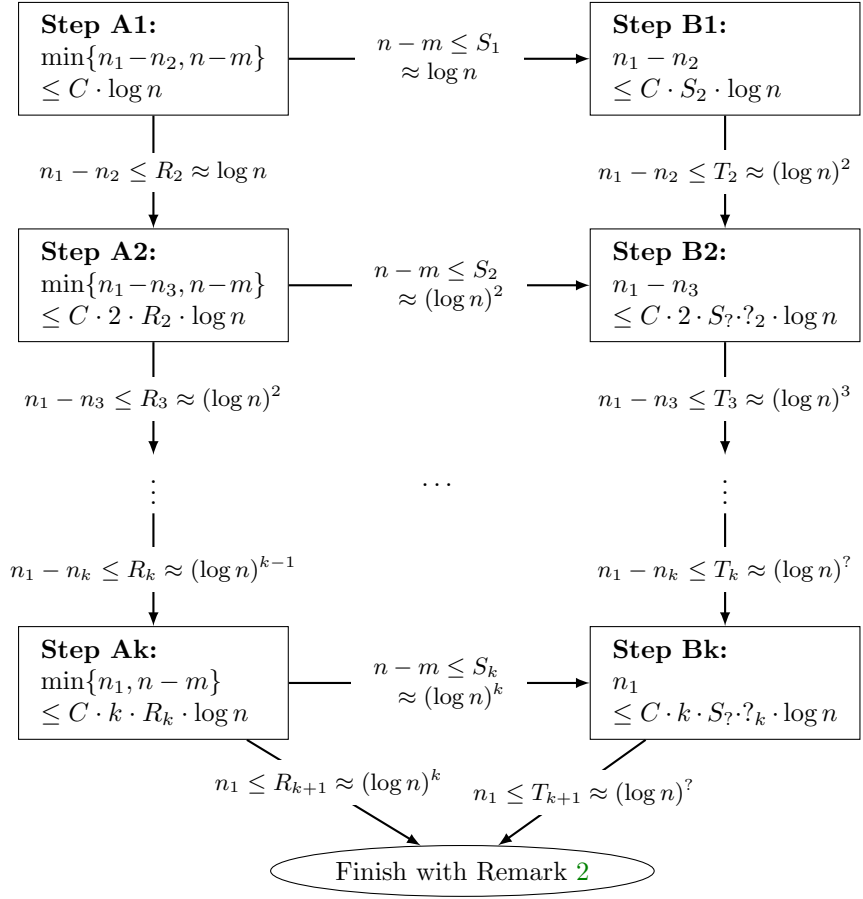
\begin{figure}[h]
\begin{tikzpicture}[auto,node distance=1.5cm]

\myboxt{A1}{}
{$\min \{ n_1 - n_2, n-m \}$ \\ $\leq \Ck \cdot \log n$}

\myboxt{A2}{below =of stepA1}
{$\min \{ n_1 - n_3, n-m \}$ \\ $\leq \Ck \cdot 2 \cdot R_2 \cdot \log n$}

\node[](dotsleft)[below = of stepA2] {\vdots};

\myboxt{Ak}{below =of dotsleft}
{$\min \{n_1, n-m \}$ \\ $\leq \Ck \cdot k \cdot R_k \cdot \log n$}

\myboxt{B1}{right = 4cm of stepA1}
{$n_1 - n_2$ \\ $\leq \Ck \cdot S_2 \cdot \log n$}

\myboxt{B2}{right = 4cm of stepA2}
{$n_1 - n_3$ \\ $\leq \Ck \cdot 2 \cdot S_? \cdot ?_2 \cdot \log n$}

\node[](dotsright)[below = of stepB2] {\vdots};

\myboxt{Bk}{right = 4cm of stepAk}
{$n_1$ \\ $\leq \Ck \cdot k \cdot S_? \cdot ?_k \cdot \log n$}

\node (dotsmiddle) at ($(dotsleft)!0.5!(dotsright)$) {\dots};

\node[ellipse, draw] (doneEnd) [below right = 1.5cm and 0.5cm of stepAk] 
{Finish with Remark \ref{rem:strategy}};


\pfeilrueber{stepA1}{stepB1}{$n-m \leq S_1$  \\ $\quad \approx \log n$}
\pfeilrueber{stepA2}{stepB2}{$n-m \leq S_2$  \\ $\quad \approx (\log n)^2$}
\pfeilrueber{stepAk}{stepBk}{$n-m \leq S_{k}$  \\ $\quad \approx (\log n)^{k}$}

\pfeilrunter{stepA1}{stepA2}{$n_1-n_2$ $\leq R_2 \approx \log n$}
\pfeilrunter{stepA2}{dotsleft}{$n_1-n_3$ $\leq R_3 \approx (\log n)^2$}
\pfeilrunter{dotsleft}{stepAk}{$n_1-n_{k}$ $\leq R_{k} \approx (\log n)^{k-1}$}

\pfeilrunter{stepB1}{stepB2}{$n_1-n_2$ $\leq T_2 \approx (\log n)^2 $}
\pfeilrunter{stepB2}{dotsright}{$n_1-n_3$ $\leq T_3 \approx (\log n)^3$}
\pfeilrunter{dotsright}{stepBk}{$n_1-n_k$ $\leq T_k \approx (\log n)^{?}$}

\draw [-latex,thick] (stepAk) --  node [fill=white, anchor=center, pos=0.4] {\small{\quad $n_1 \leq R_{k+1} \approx (\log n)^k$}} (doneEnd) ;
\pfeilrunter{stepBk}{doneEnd}{\quad $n_1 \leq T_{k+1} \approx (\log n)^?$}

\end{tikzpicture}
\caption{Overview of steps}
\label{fig:steps2}
\end{figure}

In order to compute the maximal such bound $n_1 \leq R_{k+1}$ or $n_1 \leq T_{k+1}$, we need to see what exactly happens when we ``walk the steps''.

\noindent\textit{Case 1:} $n_1 < n-m$ (``Walking down the left side'').

\noindent In this case we start with $R_1 = 1$ and at each Step A$\ell$ we compute the next bound as described in \eqref{eq:boundStepAl}, namely by
\[
	R_{\ell + 1} = R_{\ell} \cdot \Ck \cdot \ell \cdot \log n.
\] 
From this recursion, we immediately see that the last bound is 
\begin{equation}\label{eq:boundRk+1}
	R_{k+1} = \Ck^k \cdot k! \cdot (\log n)^k.
\end{equation}

\medskip\noindent\textit{Case 2:} $n_1 \geq n-m$ (``Crossing over'').

\noindent In this case, we cross from left to right at some point in Figure~\ref{fig:steps2}, i.e.\ we go from Step A$\ell_0$ to Step B$\ell_0$ for some $\ell_0 \in \{1\bb k\}$. This means that we get bounds in the following order: 
$R_1 = 1, R_2 \bb R_{\ell_0}, S_{\ell_0}, T_{\ell_0 + 1} \bb T_{k+1}$. 
By the same reasoning as in Case 1, the last bound on the left will be 
\[
	R_{\ell_0} = \Ck^{\ell_0-1} \cdot (\ell_0 -1)! \cdot (\log n)^{\ell_0 -1}.
\]
Then, in Step A$\ell_0$, the bound $S_{\ell_0}$ is computed in the same way, and we obtain
\[
	S_{\ell_0}
	= \Ck^{\ell_0} \cdot \ell_0! \cdot (\log n)^{\ell_0}.
\]
After that, in Step B$\ell_0$, we obtain by \eqref{eq:bound_stepBl} (or in the case $\ell_0=k$ by \eqref{eq:bound_stepBk}) the bound
\begin{align*}
	T_{\ell_0+1}
	&= \Ck \cdot \ell_0 \cdot S_{\ell_0} \cdot R_{\ell_0} \cdot \log n\\
	&= \Ck \cdot \ell_0 
		\cdot (\Ck^{\ell_0} \cdot \ell_0! \cdot (\log n)^{\ell_0}) 
		\cdot ( \Ck^{\ell_0-1} \cdot (\ell_0 -1)! \cdot (\log n)^{\ell_0 -1}) 
		\cdot \log n\\
	&= \Ck^{2\ell_0} \cdot (\ell_0 !)^2 \cdot (\log n)^{2\ell_0}.
\end{align*}
After that, for $1\leq i \leq k-\ell_0$, we have the recursion
\begin{align*}
	T_{\ell_0 + i + 1}
	&= \Ck \cdot (\ell_0 + i) \cdot S_{\ell_0} \cdot T_{\ell_0 + i} \cdot \log n\\
	&= \Ck \cdot (\ell_0 + i) \cdot (\Ck^{\ell_0} \cdot \ell_0! \cdot (\log n)^{\ell_0}) \cdot T_{\ell_0 + i} \cdot \log n \\
	&= T_{\ell_0 + i} \cdot C^{\ell_0 + 1} \cdot \ell_0! \cdot (\log n)^{\ell_0 + 1} \cdot  (\ell_0 + i).
\end{align*}
From this recursion, we obtain the last bound
\begin{align}\label{eq:Tkplus1}
	T_{k+1}
	&= T_{\ell_0+1 + (k-\ell_0)} \\
	&= T_{\ell_0+1} \cdot (C^{\ell_0 + 1} \cdot \ell_0! \cdot (\log n)^{\ell_0 + 1})^{k-\ell_0} \cdot (\ell_0+1)(\ell_0 + 2) \cdots k \nonumber\\
	&= \left(\Ck^{2\ell_0} \cdot (\ell_0 !)^2 \cdot (\log n)^{2\ell_0} \right)
		\cdot \Ck^{(\ell_0+1)(k-\ell_0)} \cdot (\ell_0!)^{k-\ell_0} \cdot (\log n)^{(\ell_0+1)(k-\ell_0)} \cdot k!/\ell_0!\nonumber\\
	&= k! 
		\cdot \Ck ^{(\ell_0+1)(k-\ell_0) + 2 \ell_0} 
		\cdot 	(\ell_0!)^{k-\ell_0 + 1}
		\cdot  (\log n)^{(\ell_0+1)(k-\ell_0) + 2 \ell_0}.\nonumber
\end{align}
This bound depends on $\ell_0$. In order to obtain an overall upper bound for all $1\leq \ell_0 \leq k$, we first compute the maximum of the exponent $(\ell_0+1)(k-\ell_0) + 2 \ell_0 = k + \ell_0(k  + 1- \ell_0)$. For fixed $k$, the quadratic function $f(\ell) = k + \ell(k  + 1 - \ell)$ has a maximum in $\ell = (k+1)/2$, and the maximum is $(k^2 +6k + 1)/4$. Thus we have 
\[
	(\ell_0+1)(k-\ell_0) + 2 \ell_0 
	\leq (k^2 +6k + 1)/4.
\]
Finding the maximum of the expression $(\ell_0!)^{k-\ell_0 + 1}$ is much harder. We bound the expression in a rough way. Note that for $1\leq \ell \leq k$, we have
\[
	(\ell!)^{k-\ell + 1}
	\leq \left( \ell ^\ell \right)^{k-\ell + 1}
	\leq \left( k ^\ell \right)^{k-\ell + 1}
	\leq k^{\ell(k-\ell + 1)}.
\]
Again, the exponent $\ell(k-\ell + 1) = \ell(k + 1 -\ell)$ is maximal in $\ell = (k+1)/2$, and the maximum is $(k^2 +2k + 1)/4$. Thus we have
\[
	(\ell_0!)^{k-\ell_0 + 1}
	\leq k^{(k^2 +2k + 1)/4}.
\]
Finally, for simplicity, we estimate $k! \leq k^k$. Then we obtain from \eqref{eq:Tkplus1} that 
\begin{align*}
	T_{k+1}
	&\leq k^k \cdot \Ck^{(k^2 +6k + 1)/4} 
		\cdot 	k^{(k^2 +2k + 1)/4}
		\cdot  (\log n)^{(k^2 +6k + 1)/4}\\
	&= \Ck ^{(k^2 +6k + 1)/4} 
		\cdot 	k^{(k^2 +2k + 5)/4}
		\cdot  (\log n)^{(k^2 +6k + 1)/4}.
\end{align*}
Thus, we have proven that no matter at which  point we cross from the left to the right (Step A${\ell_0}$ -- Step B$\ell_0$), we always end up with the bound above.
Since this bound is of course larger than the bound \eqref{eq:boundRk+1} from Case~1, we overall obtain 
\begin{equation}\label{eq:bound_after_steps}
	n_1
	\leq \Ck ^{(k^2 +6k + 1)/4} 
		\cdot 	k^{(k^2 +2k + 5)/4}
		\cdot  (\log n)^{(k^2 +6k + 1)/4}.
\end{equation}

\section{Finishing the proof}\label{sec:finish}

Now we finish the proof as announced in Remark~\ref{rem:strategy}.
Inequilality~\eqref{eq:bound_after_steps} combined with Theorem~\ref{thm:Kebli} and Lemma~\ref{lem:logyn1} yields
\begin{align}
	n \nonumber
	< 6 \cdot 10^{29} \cdot n_1^4
	&\leq 6 \cdot 10^{29} \cdot (\Ck ^{(k^2 +6k + 1)/4} 
		\cdot 	k^{(k^2 +2k + 5)/4}
		\cdot  (\log n)^{(k^2 +6k + 1)/4})^4 \\
	&\leq \Ck ^{k^2 +6k + 3} 
		\cdot 	k^{k^2 +2k + 5}
		\cdot  (\log n)^{k^2 +6k + 1}, \label{eq:n-inequ-finish}
\end{align}
where we used $6 \cdot 10^{29} \leq C^2 = (2.1 \cdot 10^{15})^2$.

Let $\eps >0$ be given. 

We want to apply Lemma~\ref{lem:ungl} to inequality~\eqref{eq:n-inequ-finish}, 
 setting $c = \Ck ^{k^2 +6k + 3} \cdot k^{k^2 +2k + 5}$ and $x=k^2 +6k + 1$. Moreover, we fix a $0 < \delta < 1$, which we will specify in a moment.

First, we compute the last bound from Lemma~\ref{lem:ungl}:
 \begin{align*}
	(2x)^{(1+\delta)x} \cdot c
	&= (2 (k^2 +6k + 1))^{(1+\delta) (k^2 +6k + 1)}
		\cdot (\Ck ^{k^2 +6k + 3} \cdot k^{k^2 +2k + 5})\\
	&\leq (16k^2)^{(1+\delta) (k^2 +6k + 1)} \cdot \Ck ^{10k^2} \cdot k^{k^2 +2k + 5} \\	
	&\leq C_1^{k^2} \cdot k^{2(1+\delta) (k^2 +6k + 1) + (k^2 + 2k + 5)},
\end{align*}
where we may have set $C_1 = 16 ^{2 \cdot 8} \cdot \Ck^{10}$. 
Now if we fix a $0 < \delta < \min\{\eps/2,1\}$, then the expression 
$k^{(3+\eps)k^2}$ grows faster than the bound $ C_1^{k^2} \cdot k^{2(1+\delta) (k^2 +6k + 1) + (k^2 + 2k + 5)}$. Therefore, there exists an effectively computable constant $C_2(\delta, \eps)$, such that
\begin{equation}\label{eq:finalbound3}
	(2x)^{(1+\delta)x} \cdot c
	\leq C_2(\delta, \eps)  k^{(3+\eps)k^2}.
\end{equation}

Next, we compute the second bound from Lemma~\ref{lem:ungl}:
\begin{align*}
	2^x \cdot c \cdot (\log c)^x
	&= 2 ^ {k^2 +6k + 1} 
		\cdot ( \Ck ^{k^2 +6k + 3} \cdot k^{k^2 +6k + 1} )
		\cdot (\log (\Ck ^{k^2 +6k + 3} \cdot k^{k^2 +2k + 5}))^{k^2 +6k + 1} \\
	&\leq C_3^{k^2} \cdot  k^{k^2 +6k + 1}
		\cdot ((k^2 +6k + 1) \log \Ck + (k^2 +2k + 5) \log k)^{k^2 +6k + 1}\\
	& \leq C_4^{k^2} \cdot  k^{k^2 +6k + 1}
		\cdot (k^2 \log k)^{k^2 +6k + 1}\\
	&= C_4^{k^2} \cdot  (k^3 \log k)^{k^2 +6k + 1},
\end{align*}
where $C_3, C_4$ are effectively computable constants (similarly to how we obtained $C_1$ in the previous computation).
Again, since for any fixed $\eps>0$, the expression $k^{(3+\eps)k^2}$ grows faster than the bound $ C_4^{k^2} \cdot  (k^3 \log k)^{k^2 +6k + 1}$, there exists an effectively computable constant $C_5(\eps)$, such that
\begin{equation}\label{eq:finalbound2}
	2^x \cdot c \cdot (\log c)^x
	\leq C_5(\eps)  k^{(3+\eps)k^2}.
\end{equation}

Finally, we consider the first bound from Lemma~\ref{lem:ungl}.
Since we have fixed $\delta$, it is clear that there exists an effectively computable constant $C_6(\delta,\eps)$, such that
\begin{equation}\label{eq:finalbound1}
	\exp (\exp ( ( 1+ \delta^{-1})^2) )
	\leq C_6(\delta,\eps) k^{(3+\eps)k^2}.
\end{equation}

We set
\[
	C(\eps) = \max\{
		C_2(\delta, \eps), C_5(\eps), C_6(\eps, \delta)
	\}.
\]
Now an application of Lemma~\ref{lem:ungl} to \eqref{eq:n-inequ-finish}, together with \eqref{eq:finalbound1}, \eqref{eq:finalbound2} and \eqref{eq:finalbound3}, yields
\begin{equation}\label{eq:nfinalbound}
	n
	\leq C(\eps) k^{(3+\eps)k^2}.
\end{equation}
Finally, we can bound $y^a$ by
\begin{align*}
	\log y^a
	&= \log (F_n + F_m)
	< \log (2 F_n)
	< \log (2 \alpha^n)
	= \log 2 + n \log \alpha \\
	&\leq \log 2 + C(\eps) k^{(3+\eps)k^2} \log \alpha.
\end{align*}
This implies
\[
	\log y^a
	\leq C(\eps) k^{(3+\eps)k^2}
\]
and we have proven Theorem~\ref{thm:main}. \qed
 
\bibliographystyle{habbrv}
\bibliography{Literatur_Dioph}

\begin{thebibliography}{10}
\expandafter\ifx\csname url\endcsname\relax
  \def\url#1{\texttt{#1}}\fi
\expandafter\ifx\csname doi\endcsname\relax
  \def\doi#1{\burlalt{doi:#1}{http://dx.doi.org/#1}}\fi
\expandafter\ifx\csname urlprefix\endcsname\relax\def\urlprefix{URL: }\fi
\expandafter\ifx\csname href\endcsname\relax
  \def\href#1#2{#2}\fi
\expandafter\ifx\csname burlalt\endcsname\relax
  \def\burlalt#1#2{\href{#2}{#1}}\fi

\bibitem{BravoLuca2016}
J.~J. Bravo and F.~Luca.
\newblock On the {D}iophantine equation {$F_n+F_m=2^a$}.
\newblock {\em Quaest. Math.}, 39(3):391--400, 2016.
\newblock \doi{10.2989/16073606.2015.1070377}.

\bibitem{BugeaudLucaMignotteSiksek2007}
Y.~Bugeaud, F.~Luca, M.~Mignotte, and S.~Siksek.
\newblock Perfect powers from products of terms in {L}ucas sequences.
\newblock {\em J. Reine Angew. Math.}, 611:109--129, 2007.
\newblock \doi{10.1515/CRELLE.2007.075}.

\bibitem{BugeaudMignotteSiksek2006}
Y.~Bugeaud, M.~Mignotte, and S.~Siksek.
\newblock Classical and modular approaches to exponential {D}iophantine
  equations. {I}. {F}ibonacci and {L}ucas perfect powers.
\newblock {\em Ann. of Math. (2)}, 163(3):969--1018, 2006.
\newblock \doi{10.4007/annals.2006.163.969}.

\bibitem{Cordwell:2018}
K.~Cordwell, M.~Hlavacek, C.~Huynh, S.~J. Miller, C.~Peterson, and Y.~N.~T. Vu.
\newblock Summand minimality and asymptotic convergence of generalized
  {Z}eckendorf decompositions.
\newblock {\em Res. Number Theory}, 4(4):Paper No. 43, 27, 2018.
\newblock \doi{10.1007/s40993-018-0137-7}.

\bibitem{KebliKihelLaroneLuca2021}
S.~Kebli, O.~Kihel, J.~Larone, and F.~Luca.
\newblock On the nonnegative integer solutions to the equation {$F_ n\pm F_m =
  y^a$}.
\newblock {\em J. Number Theory}, 220:107--127, 2021.
\newblock \doi{10.1016/j.jnt.2020.08.004}.

\bibitem{KihelLarone2021}
O.~Kihel and J.~Larone.
\newblock On the nonnegative integer solutions of the equation {$F_n\pm
  F_m=y^a$}.
\newblock {\em Quaest. Math.}, 44(8):1133--1139, 2021.
\newblock \doi{10.2989/16073606.2020.1775155}.

\bibitem{Luca:2000}
F.~Luca.
\newblock Distinct digits in base {$b$} expansions of linear recurrence
  sequences.
\newblock {\em Quaest. Math.}, 23(4):389--404, 2000.
\newblock \doi{10.2989/16073600009485986}.

\bibitem{Luca:2017}
F.~Luca, A.~Montejano, L.~Szalay, and A.~Togb\'{e}.
\newblock On the {$X$}-coordinates of {P}ell equations which are tribonacci
  numbers.
\newblock {\em Acta Arith.}, 179(1):25--35, 2017.
\newblock \doi{10.4064/aa8553-2-2017}.

\bibitem{LucaPatel2018}
F.~Luca and V.~Patel.
\newblock On perfect powers that are sums of two {F}ibonacci numbers.
\newblock {\em J. Number Theory}, 189:90--96, 2018.
\newblock \doi{10.1016/j.jnt.2018.02.003}.

\bibitem{Matveev2000}
E.~M. Matveev.
\newblock An explicit lower bound for a homogeneous rational linear form in the
  logarithms of algebraic numbers. {II}.
\newblock {\em Izv. Math.}, 64(6):1217--1269, 2000.
\newblock \doi{10.1070/im2000v064n06abeh000314}.

\bibitem{Ziegler2022}
V.~Ziegler.
\newblock Sums of fibonacci numbers that are perfect powers.
\newblock {\em Quaest. Math.}, 0(0):1--26, 2022,
  \burlalt{https://doi.org/10.2989/16073606.2022.2109220}{http://arxiv.org/abs/https://doi.org/10.2989/16073606.2022.2109220}.
\newblock \doi{10.2989/16073606.2022.2109220}.

\end{thebibliography}

\end{document}